\documentclass[11pt]{amsart}

\usepackage{amssymb,epsf,psfrag}
\usepackage{wasysym}
\usepackage{graphicx}
\usepackage{color}
\usepackage{amsfonts}
\usepackage{enumerate}
\usepackage{latexsym}
\usepackage{epsfig}
\usepackage{amsthm}
\usepackage{amsmath}
\usepackage{latexsym}
\usepackage[all]{xy}

\usepackage[
            breaklinks=true,
            colorlinks=true]{hyperref}
\usepackage[active]{srcltx}

\makeatletter\@addtoreset{figure}{section}\makeatother
\makeatletter\@addtoreset{table}{section}\makeatother

\newtheorem{theorem}{Theorem}[]
\newtheorem{prop}[theorem]{Proposition}

\newtheorem{cor}[theorem]{Corollary}

\newcommand{\del}{\partial}

\newcommand{\calC}{\mathcal C}

\newcommand{\RR}{\mathbb R}

\newcommand{\op}[1]{\!\!\mathop{\rm ~#1}\nolimits}

\begin{document}

\title[Complete Hamiltonian manifolds]{${\rm L}^2$\--cohomology and complete Hamiltonian manifolds}

\author{}
\author{Rafe Mazzeo${}^1$\,\,\,\,\,\,\,\, \'Alvaro Pelayo${}^2$\,\,\,\,\,\,\,\, Tudor S. Ratiu${}^3$}

\addtocounter{footnote}{1} 
\footnotetext{ Partially supported by NSF grant DMS-1105050}

\addtocounter{footnote}{1} 

\footnotetext{ Partially 
supported by  NSF Grant DMS-0635607,
 NSF CAREER Award  1055897, Spanish Ministry Grant MTM 2010-21186-C02-01, and the MSRI 2010-11 program \textit{Symplectic and Contact
Geometry and Topology}.
\addtocounter{footnote}{1} }

\footnotetext{Partially supported by the government grant of the 
Russian Federation for support of research projects implemented by 
leading scientists, Lomonosov Moscow State University, under  
agreement No. 11.G34.31.0054, by  Swiss NSF grant 200021-140238,
and the 2010-11 MSRI program \textit{Symplectic and Contact
Geometry and Topology}.
\addtocounter{footnote}{1} }

\date{}

\maketitle

\begin{abstract}
A classical theorem of Frankel for compact K\"ahler manifolds states that
a K\"ahler $S^1$-action is Hamiltonian if and only if it has fixed points.
We prove a metatheorem which says that when Hodge theory holds on 
non-compact manifolds,
then Frankel's theorem still holds. Finally, we present several concrete situations
in which the assumptions of the metatheorem hold. 
\end{abstract}

\section{The Classical Frankel Theorem} \label{intro}

An $S^1$-action on  a symplectic manifold $(M,\omega)$ is \emph{Hamiltonian} if there exists a smooth map, the \emph{momentum  map}, 
$$\mu \colon M \to (\mathfrak{s}^1)^* \simeq \mathbb{R}$$ into the dual 
$(\mathfrak{s}^1)^*$ of the Lie algebra $\mathfrak{s}^1 \cong \mathbb{R}$ 
of $S^1$,  such that $$\mathbf{i}_{\xi_M}\omega : = \omega(\xi_M, \cdot) = 
\mathbf{d} \mu,$$ for some generator $\xi$ of $\mathfrak{s}^1$, that is, the $1$-form $\mathbf{i}_{\xi_M}\omega$ is exact. 
Here $\xi_M$ is the vector field on $M$ whose flow is given by $\mathbb{R} \times M \ni (t,m) \mapsto e^{\rm it \xi} 
\cdot m \in M$, where the dot denotes the $S^1$-action on $M$. If $(M,\omega)$ is connected, compact and 
\emph{K\"ahler}, the following result of T. Frankel is well-known: 
\\
\vspace{-2.5mm}
\\
{\bf Frankel's Theorem} (\cite{Frankel1959}).
 \emph{Let $M$ be a compact connected K\"ahler
manifold admitting an $S^1$-action  preserving the K\"ahler
structure. If the $S^1$-action has fixed points, then the
action is Hamiltonian.}

\medskip

This theorem generalizes in various ways; for example, the $S^1$-action may be replaced by a $G$-action, where $G$
is any compact Lie group and the K\"ahler structure may be weakened to a symplectic structure. The purpose of this 
paper is to generalize Frankel's theorem to certain noncompact complete Riemannian manifolds. More specifically, 
we describe a set of hypotheses under which the proof in the compact case can be generalized. This relies on
the existence of a Hodge decomposition on $1$-forms.

\section{Hodge Decomposition implies Frankel's Theorem} \label{fs}
Let $(M,\omega)$ be a symplectic manifold. The triple $(\omega,g,\mathbf{J})$ is a \emph{compatible triple} on $(M,\omega)$ 
if $g$ is a Riemannian metric and $\mathbf{J}$ is an almost complex structure such that $g(\cdot,\cdot)=\omega(\cdot,\mathbf{J}\cdot)$. 
Denote by ${\rm d}V_g$ the measure associated to the Riemannian volume.

Let $G$ be a connected Lie group with Lie algebra $\mathfrak{g}$ acting on $M$ by symplectomorphisms, i.e., diffeomorphisms 
which preserve the symplectic form. We refer to $(M,\, \omega)$ as a \emph{symplectic $G$\--manifold}. Any element 
$\xi \in \mathfrak{g}$ generates a vector field $\xi_M$ on $M$, called the \emph{infinitesimal generator}, given by 
$$\xi_M(x):= \left.\frac{{\rm d}}{{\rm d}t}\right|_{t=0}  \op{exp}(t\xi)\cdot x.$$ 

The  $G$-action on $(M,\omega)$  is said to be \emph{Hamiltonian} if there exists a smooth 
equivariant map $\mu  \colon M \to \mathfrak{g}^*$, called the \emph{momentum map}, such that for
all $\xi \in \mathfrak{g}$ we have  $$\mathbf{i}_{\xi_M} \omega : = \omega(\xi_M, \cdot) = \mathbf{d} \langle \mu, 
\xi \rangle,$$ where $\left\langle \cdot , \cdot \right\rangle : \mathfrak{g}^\ast \times \mathfrak{g} 
\rightarrow \mathbb{R}$ is the duality pairing. For example, if $G\simeq (S^1)^k$, $k \in \mathbb{N}$, is a torus, the existence
of such a map $\mu$ is equivalent to the exactness of the one-forms $\mathbf{i}_{\xi_M}\omega$ for all $\xi \in \mathfrak{g}$. 

In this case the obstruction of the action to being Hamiltonian lies in the first de Rham cohomology
group of $M$. The simplest example of a $S^1$-Hamiltonian action is rotation of the sphere $S^2$ about the polar axis. 
The flow lines of the infinitesimal generator defining this action are the latitude circles. 

Denote by ${\rm L}^2_\lambda$ the Hilbert space of square integrable functions relative 
to a given measure ${\rm d}\lambda$ on $M$, and write the associated norm either on functions or $1$-forms as 
$\|\cdot \|_{{\rm L}^2_\lambda}$. This measure determines a formal 
adjoint $\delta_\lambda$ of the de Rham differential.
A ${\rm L}^2_\lambda$ $1$-form $\omega$ is called \emph{$\lambda$\--harmonic} if it is in the common null space
of $\mathbf{d}$ and $\delta_\lambda$. 

\begin{theorem}\label{thm_gen} 
Let $G$ be a compact connected Lie group acting on the symplectic 
manifold $(M,\omega)$, with $(\omega, g, \mathbf{J})$ a $G$-invariant 
compatible triple. Suppose, in addition, that 
${\rm d}\lambda = f {\rm d}V_g$ 
is a $G$-invariant measure on $M$, where $f$ is 
smooth and bounded. Suppose that $\|\xi_M\|_{{\rm L}^2_\lambda} < \infty$ for all $\xi\in \mathfrak{g}$. Assume that every smooth 
closed $1$-form $\omega$ in ${\rm L}^2_ \lambda$ decomposes as an 
${\rm L}^2_ \lambda$-orthogonal sum 
$\mathbf{d} f + \chi$, where $f, \mathbf{d}f \in {\rm L}^2_\lambda$, 
 $\chi \in {\rm L}^2_ \lambda$ is $\lambda$-harmonic, 
and that each cohomology class in $H^1(M)$ has a unique 
$\lambda$-harmonic  representative
in ${\rm L}^2_ \lambda$. If $\mathbf{J}$ preserves the space of ${\rm L}^2_\lambda$ harmonic one-forms and the $G$-action has fixed points 
on every connected component, then the action is Hamiltonian. 
\end{theorem}
\begin{proof}
The proof extends Frankel's method \cite{Frankel1959}.  For clarity, we divide the proof into several steps.
\medskip

\noindent \textbf{Step 1} (Infinitesimal invariance of $\lambda$-harmonic $1$-forms). We show first that if $\alpha \in 
\Omega^1(M)$ is harmonic and $\| \alpha\|_{{\rm L}^2_\lambda} < \infty$, then $\mathcal L_{\xi_M} \alpha = 0$.
This is standard in the usual setting, but requires checking here since we have that $\delta_\lambda \alpha = 0$ rather than 
$\delta \alpha = 0$. 

If $\varphi$ is an isometry of $(M,g)$ and preserves the measure $\operatorname{d}\!\lambda$, then 
\begin{equation}
\label{useful_formula}
\varphi^\ast \left(\left\langle \! \left\langle \nu, \rho \right\rangle \! \right\rangle \operatorname{d}\!\lambda\right) = 
\left\langle \! \left\langle \varphi ^\ast \nu, \phi^\ast \rho \right\rangle \! \right\rangle
\operatorname{d}\!\lambda \nonumber
\end{equation}
for any $\nu, \rho \in \Omega^1(M)$, where $\left\langle \! \left\langle \nu, \rho \right\rangle \! \right\rangle $ dentes
the pointwise inner product of $\nu$ and $\rho$ on $M$. 

Next, denote by $\Phi: G \times M \rightarrow M$ the $G $-action and $F_t := \Phi_{\exp(t \xi)} $ the flow of $\xi_M$. 
Since $\mathbf{d} \alpha = 0 $ it follows that $$\mathbf{d}F  _t^\ast \alpha = F _t^\ast \mathbf{d}\alpha = 0.$$
In addition, since $F_t^\ast$ commutes with $\lambda$, we also have 
$$\delta_\lambda F_t^\ast \alpha = F_t^\ast \delta_\lambda \alpha = 0.$$
Hence if $\alpha$ is harmonic, then so is $F _t^\ast \alpha$.

However, because $F_t$ is isotopic to the identity, 
$$
[F_t^\ast \alpha] = F_t^\star[ \alpha] = [\alpha]
$$ 
in $\op{H}^1(M, \mathbb{R})$, where $F_t^\star$ is the map on cohomology
induced by $F_t$.
This implies that $F _t^\ast \alpha = \alpha$ since this cohomology class contains only one harmonic representative.
Taking the $t $-derivative yields 
$$
{\mathcal L}_{ \xi_M} \alpha =0,
$$ 
as required.
\smallskip

\noindent\textbf{Step 2} (Using the existence of fixed points). Define 
$$\xi_M^\flat: = g( \xi_M, \cdot ) \in \Omega^1(M).$$ 
If $\alpha\in \Omega^1(M)$ is harmonic and $\| \alpha\|_{ {\rm L}^2_{ \lambda}} < \infty$,
it follows from Step 1 that $$0 = {\mathcal L}_{\xi_M} \alpha = \mathbf{i}_{\xi_M} \mathbf{d} \alpha + 
\mathbf{d}\mathbf{i}_{\xi_M} \alpha = \mathbf{d}\mathbf{i}_{\xi_M} \alpha.$$ Thus $\alpha(\xi_M)$ is constant on 
each connected component of $M$. Now, $\xi_M$ vanishes on the fixed point set of $G$, and each component
of $M$ contains at least one such point. Thus $\alpha(\xi_M) \equiv 0$ on $M$, whence 
\begin{equation}
\label{vanishing}
\left\langle \xi_M^\flat, \alpha \right\rangle_{\op{L}^2_\lambda}
= \varint_M \alpha(\xi_M) \operatorname{d}\! \lambda = 0 \nonumber
\end{equation}
for any harmonic one-form $\alpha$ satisfying $\| \alpha\|_{{\rm L}^2_{ \lambda}} < \infty$. 
\smallskip

\noindent\textbf{Step 3} (Hodge decomposition). 
Since $$\mathbf{d}\mathbf{i}_{\xi_M} \omega= {\mathcal L}_{\xi_M}\omega = 0$$ and 
$\|\mathbf{i}_{\xi_M} \omega\|_{{\rm L}^2_{\lambda}} < \infty$, our hypothesis implies that  
\begin{equation}
\label{Hodge_in_this_case}
\mathbf{i}_{\xi_M} \omega = \mathbf{d} f^ \xi + \chi^ \xi, \nonumber
\end{equation}
where $f^ \xi \in \op{C}^{\infty}(M) $, $\chi^ \xi \in \Omega^1(M)$ is $\lambda$-harmonic and 
$\|\mathbf{d} f ^\xi\|_{{\rm L}^2_{ \lambda}}, \|\chi^\xi\|_{{\rm L}^2_{ \lambda}} < \infty$. 

We now prove that $\chi^\xi=0$. If $ \alpha \in \Omega^1(M) $ is any harmonic one-form with 
$\|\alpha\|_{{\rm L}^2_{ \lambda}} < \infty$, then 
\begin{equation*}
\left\langle \mathbf{i}_{\xi_M} \omega, \alpha \right\rangle_{{\rm L}^2_ \lambda}
= \left\langle \xi_M^\flat, \mathbf{J}\alpha \right\rangle_{\op{L}^2_\lambda} = 0
\end{equation*}
by Step 2 since $\mathbf{J} \alpha$ is harmonic (by the hypotheses of
the theorem). In particular,  since $$\xi_M^\flat =  \mathbf{J}\mathbf{d}f^\xi + 
\mathbf{J}\chi^\xi$$ and $\xi_M^\flat$ is also orthogonal to the first term on the right, we conclude that $\chi^\xi = 0$. 

The conclusion of this step is $$\mathbf{i}_{\xi_M} \omega = \mathbf{d}f^ \xi $$ for 
any $\xi\in \mathfrak{g}$; note that both sides of this identity are linear in $\xi$. 
\smallskip

\noindent\textbf{Step 4} (Equivariant momentum map). Using a basis $\{e_1, \ldots, e_r\}$ of $\mathfrak{g}$, we 
define $\mu:M \rightarrow \mathfrak{g}^\ast$ by 
$$
\mu^ \xi := \xi^1f^{e_1}+ \cdots + \xi^rf^{e_r}, \quad \mbox{where}\quad \xi= \xi^1e_1+ \cdots + \xi^re_.
$$ 
Clearly, $$\mathbf{i}_{\xi_M} \omega = \mathbf{d}\mu^ \xi,$$ so $\mu$  is a momentum map of the $G $-action. 
Since $G$ is compact, one can average $\mu$ in the standard way (see, e.g., 
\cite[Theorem 11.5.2]{MaRa2003}) to obtain an equivariant momentum map. 
This completes the proof of the theorem.
\end{proof}

\section{Applications}
We now discuss several different criteria which ensure that the results of the last section can be applied. 
The first is the classical setting of `unweighted' ${\rm L}^2$ cohomology, which is the cohomology of
the standard Hilbert complex of ${\rm L}^2$ differential forms on a complete Riemannian manifold. The
existence of a strong Kodaira decomposition is known in many instances, and we present a few
examples.  We then discuss two other criteria, the first by Ahmed and Stroock and the second by 
Gong and Wang, which allow one to prove a similar strong Kodaira decomposition for forms
which are in ${\rm L}^2$ relative to some weighted measure. We present some examples to which these
criteria apply. Finally, we recall some well-known facts
about the Hodge theory on spaces with `fibered boundary' geometry; these include asymptotically
conical spaces, as well as the important classes of ALE/ALF/... gravitational instantons. Many
of these spaces admit circle actions. 

\subsection{Unweighted ${\rm L}^2$ cohomology} 
The nature of the Kodaira decomposition and ${\rm L}^2$ Hodge theory on a complete manifold relative 
to the standard volume form is now classical. An account may be found in de Rham's book \cite{dR}; 
see also \cite{Car2002}. 
\begin{theorem}
\label{GpTr}
If $(M^n,g)$ is a complete Riemannian manifold and $0 \leq k \leq n$, then the  following conditions are equivalent:
\begin{itemize}
\item[(i)] ${\rm Im}(\mathbf{d} \delta + \delta\mathbf{d})=
(\mathcal{H}_2^k(M))^{\perp}$;
\item[(ii)] There is an ${\rm L}^2$-orthogonal decomposition
\[
{\rm L}^2(M,\Lambda^k)={\rm Im}\, \mathbf{d}\oplus {\rm Im} \, \delta \oplus \mathcal{H}_2^k(M);
\]
\item[(iii)] ${\rm Im}\, \mathbf{d}$ and ${\rm Im}\, \delta$ are closed in ${\rm L}^2(M,\Lambda^k)$;
\item[(iv)] The quotients
\[
\overline{ \mathrm{ran}\, \mathbf{d}} \big/ \mathrm{ran}\, \mathbf{d} = 0
\]
in ${\rm L}^2(M, \Lambda^k)$ and ${\rm L}^2(M, \Lambda^{n-k})$. 
\end{itemize}
If the smooth form $\alpha \in \Omega^k(M)$ decomposes as 
$\mathbf{d}\beta + \delta_{\mu_g} \gamma + \chi$, then 
$\beta\in\Omega^{k-1}(M)$, $\gamma \in\Omega^{k+1}(M)$, and 
$\chi\in \Omega^k(M) \cap \mathcal{H}^2_k(M)$ are all smooth.
\end{theorem}

Theorems \ref{thm_gen} and \ref{GpTr} imply the following result.
\begin{cor}
 \label{nicecor}
Let $G$ be a compact Lie group which acts isometrically on $(M, \omega)$, a $2n$-dimensional complete connected 
K\"ahler manifold, and suppose that any one of the conditions {\rm (i) - (iv)} of Theorem \ref{GpTr} holds. 
If the infinitesimal generators of the action all lie in $\op{L}^2_{\omega^n}$ and if the $G$-action 
has fixed points, then it is Hamiltonian.
\end{cor}
The only point to note is that since $M$ is K\"ahler, the complex structure $\mathbf{J}$ 
preserves the space of harmonic forms \cite[Cor 4.11, Ch. 5]{Wells2008}. 

\subsection{Examples}
There are many common geometric settings where the result above applies. We recall a few of these here.

\medskip

\noindent{\bf Conformally compact manifolds:}  A complete manifold $(M,g)$ is called conformally
compact if $M$ is the interior of a compact manifold with boundary $\bar{M}$ and $g$ can be written
as $\rho^{-2} \bar{g}$, where $\rho$ is a defining function for $\del \bar{M}$ (i.e., $\del\bar{M} = \{\rho = 0\}$
and ${\rm d} \rho \neq 0$ there) and $\bar{g}$ is a metric
on $\bar{M}$ which is non-degenerate and smooth up to the boundary.
The sectional curvatures of $g$ become isotropic near any point $p \in \del \bar{M}$, with common
value $-|{\rm d}\rho|^2_{\bar{g}}$. If this value is constant along the entire boundary, then $(M,g)$ is called
asymptotically hyperbolic. 

An old well-known result \cite{Maz-Hodge} states that if $n = \dim M \neq 3$ (automatic if $M$
is symplectic), then the conditions of Theorem \ref{GpTr} are satisfied when $k = 1$, and hence 
Corollary \ref{nicecor} holds. There are now much simpler proofs of this result; see \cite{CarSurvey}. 

As explained in \cite{CarSurvey}, the conditions of Theorem \ref{GpTr} are invariant under
quasi-isometry, which means that we obtain a similar result for any symplectic manifold quasi-isometric 
to a conformally compact space.  This allows us, in particular, to 
substantially relax the regularity conditions on $\rho$ and $\bar{g}$ in this definition.

There is an interesting generalization of this to the set of complete edge metrics. The geometry here is
a bit more intricate; as before, $M$ is the interior of a smooth manifold with boundary. Now, however,
the boundary $\del \bar{M}$ is assumed to be the total space of a fibration over a compact smooth
manifold $Y$ with compact fiber $F$. We can use local coordinates $(x,y,z)$ near a point of the boundary 
where $x$ is a boundary defining function, $y$ is a set of coordinates on $Y$ lifted to $\del \bar{M}$
and then extended inward, and $z$ is a set of functions which restrict to coordinates on each fiber $F$. 
A metric $g$ on $M$ is called a \textit{complete edge metric} if in each such local coordinate system it takes the form 
\begin{multline*}
g = \frac{\mathbf{d}x^2 + \sum a_{0 \alpha}(x,y,z) \mathbf{d}x\mathbf{d}y_\alpha + \sum a_{\alpha \beta} a_{\alpha \beta}(x,y,z) \mathbf{d}y_\alpha \mathbf{d}y_\beta}{x^2} \\
+ \sum b_{0\mu} \frac{\mathbf{d}x}{x}\mathbf{d}z_\mu + b_{\alpha \mu} \frac{\mathbf{d}y_\alpha}{x} \mathbf{d}z_\mu + 
b_{\mu \nu} \mathbf{d}z_\mu \mathbf{d}z_\nu. 
\end{multline*}
The prototype is the product $X \times F$ where $X$ is a conformally compact manifold and $F$ is a compact smooth
manifold, or more generally, a manifold which fibers over a neighborhood of infinity in a conformally compact space $X$ 
with compact smooth fiber $F$.

The analytic techniques developed in \cite{Maz-edge} generalize those in \cite{Maz-Hodge} and show that
if $(M,g)$ is a space with a complete edge metric, and if $\dim Y \neq 2$, then the Hodge Laplace
operator on $1$-forms is closed. 

\medskip

\noindent{\bf Surfaces of revolution:}  
A Riemannian surface $(M,g)$ which admits an isometric $S^1$ action must be a surface of revolution, hence in
polar coordinates, 
\[
g = \mathbf{d}r^2 + f(r)^2 \mathbf{d}\theta^2,
\] 
where $\theta \in S^1$ and either $f > 0$ on $(0,\infty)$ and is a function of $r^2$ (i.e., its Taylor expansion
near $r=0$ has only even terms) which vanishes at $r=0$, or else $f$ is strictly positive on all of $\mathbb R$.  
In the first case, $M \cong \mathbb R^2$, while in the second, $M \cong S^1 \times \mathbb R$. 

The symplectic form is $\omega = f(r) \mathbf{d}r \wedge \mathbf{d}\theta$, so the action is generated by the vector field 
$\del_\theta$. Since $\mathbf{i}_{\partial/\partial\theta}\omega=
-f(r)\mathbf{d}r$, one of the basic hypotheses becomes
\begin{equation}
\|\mathbf{i}_{\del_\theta}\omega \|^2_{\op{L}^2_{\omega}}  
= \varint_{0}^{2\pi}\varint_{0}^{\infty}f(r)^3\op{d}\!r\op{d}\!\theta  =2\pi \varint_0^\infty f(r)^3 \op{d}\!r < \infty.
\label{normdtheta}
\end{equation}
\begin{prop}\cite[Theorem 1.2]{Tr2009} If $M \cong \RR^2$ and $f \leq C r^{-k}$ for some $k > 1/3$, then the range of the Hodge Laplace
operator on $1$-forms is closed.
\label{Tro}
\end{prop}
With these hypotheses, we can then apply Corollary \ref{nicecor} as before.  

It is worth contrasting Proposition \ref{Tro} with the well-known criterion of McKean \cite{McK}. This states that if $(M^2,g)$ 
is simply connected \emph{and} has Gauss curvature $K_g \leq -1$, then the ${\rm L}^2$ spectrum of the Laplacian on functions 
is contained in $[1/4, \infty)$. The spectrum of the Laplacian on $2$-forms is the same, and using a standard Hodge-theoretic 
argument, the spectrum of the Laplacian on $1$-forms is contained in $\{0\} \cup [1/4,\infty)$. Thus this curvature
bound would also guarantee the conclusion of Theorem \ref{GpTr}.  Now, $$K_g = - f''(r)/f(r) \leq -1$$ is the same as 
$$f''(r) \geq f(r).$$ Using this and the initial condition $f(0)=0$, it is not hard to show that $f$ must grow exponentially 
as $r \to \infty$, so that \eqref{normdtheta} cannot hold.  In other words, McKean's condition is useless for our
purposes. 

Of course, if the hypotheses of Proposition \ref{Tro} hold, then we do not need to apply these Hodge-theoretic
arguments since the momentum map of this circle action is given by any function $\mu(r)$ satisfying 
$$\mu'(r) = -f(r).$$ 

\medskip

\noindent{\bf Compact stratified spaces:}   Although it is outside the framework of complete manifolds, there is another
class of spaces to which these results may be applied. These are the smoothly stratified spaces with iterated edge metrics.
These include, at the simplest level, spaces with isolated conic singularities or simple edge singularities. More general
spaces of this type are obtained recursively, by using spaces such as these as cross-sections of cones, and these
cones can vary over a smooth base.  Hodge theory on such spaces was first considered by Cheeger \cite{Cheeger};
the recent papers \cite{ALMP1}, \cite{ALMP2} provide an alternate approach and generalize the spaces to allow ones
for which it is necessary to impose boundary conditions along the strata.  A complete Hodge theory is available, 
cf.\ the papers just cited.  One important way that such spaces might arise in our setting is if the group $G$ acts symplectically 
on a compact smooth manifold $M'$, but 
$G$ commutes with the symplectic action by another group $K$.  Then the action of $G$ descends to 
the quotient $M = M'/K$, and this latter space typically has precisely the stratified structure and iterated
edge metric as described above .

\subsection{Ahmed-Stroock conditions.}  
Under certain rather weak requirements on the geometry of $(M,g)$ and an auxiliary measure 
$${\rm d}\lambda = e^{-U} {\rm d}V_g,$$
Ahmed and Stroock \cite[\S6]{AhSt2000} have proved a Hodge-type decomposition. In the theorem below and the rest of
the paper, $\Delta f : = \operatorname{div} \nabla f = 
- \delta \mathbf{d} f$ is 
the usual Laplacian on functions and $\nabla^2f : = \operatorname{Hess} f$
is the Hessian of $f$, i.e., the second covariant derivative of $f$.

\begin{theorem}[\cite{AhSt2000}] Assume that $(M,g)$ is complete and
\begin{itemize} 
\item[$\bullet$]  $ \mathrm{Ric}_g \geq - \kappa_1$; 
\item[$\bullet$] the curvature operator is bounded above, i.e., 
$\left\langle \! \left\langle R \alpha, \alpha \right\rangle \! \right\rangle 
\leq \kappa_2 \| \alpha\|_{ {\rm L}^2}$ for all $\alpha \in \Omega^2(M) $, where $\kappa_1, \kappa_2 \geq 0$.  
\end{itemize}
Suppose further that $U$ is a smooth nonnegative proper function on $M$ which satisfies
\begin{itemize}
\item $\Delta U \leq C(1+U) $ and $\|\nabla U\|^2 \leq C\operatorname{e}^{ \theta U}$ for some $C<\infty$ and $\theta\in (0,1)$;
\item $\varepsilon U^{1+ \varepsilon} \leq 1+\| \nabla U\|^2$ for some $\varepsilon > 0$;
\item $\left\langle \! \left\langle v, (\nabla^2 U)(v) \right\rangle \! \right\rangle \geq -B\|v\|^2$ for every $x \in M $ and 
$v \in \operatorname{T}_xM$, where $B<\infty$. 
\end{itemize} 
Write $\delta_\lambda$ for the adjoint of $\mathbf{d}$ relative to 
${\rm d}\lambda = e^{-U} {\rm d}V_g$ and ${\rm L}^2_\lambda$ for the 
associated Hilbert space. Note that since $U\geq 0$, $\lambda$ is bounded.

Then 
\begin{enumerate}
\item[{\rm (1)}] \cite[Theorem 5.1]{AhSt2000} There is a strong Hodge decomposition on $1$-forms. In particular, if 
$\alpha \in {\rm L}^2_\lambda \Omega^1 \cap \calC^\infty$ is closed, then $\alpha = \mathbf{d}f + \chi $, where 
$f \in {\rm L}^2_\lambda \cap \calC^\infty$ and $\chi \in \mathcal H_\lambda^1$.
\item[{\rm (2)}] \cite[Theorem 6.4]{AhSt2000} Each class $[\alpha]\in \operatorname{H} ^1(M, \mathbb{R})$ has a unique 
representative in $\mathcal H_\lambda^1$. 
\end{enumerate}
\label{AScond}
\end{theorem}

\begin{cor}\label{Stroock_conditions}
Assume that $M$ is symplectic and that $(g,\omega, \mathbf{J})$ are a $G$-invariant compatible triple,
and that $U$ is also $G$-invariant. If the hypotheses of Theorem \ref{AScond} all hold,  
$\mathbf{J}\mathcal{H}^1_\lambda \subset \mathcal{H}^1_\lambda $,  and if the $G$-action has fixed points, 
then it is Hamiltonian.
\end{cor}

\subsection{Gong-Wang conditions} There are other conditions, discovered by Gong and Wang,  
which lead to a strong Hodge decomposition. 

\begin{theorem}[\cite{GoWa2004}]
Let $G$ act on the noncompact symplectic manifold $(M, \omega)$, and suppose that $(\omega, g, \mathbf{J})$ 
is a $G$-invariant compatible triple. Assume that ${\rm d}\lambda = e^V {\rm d}V_g$ is also $G$-invariant and has finite 
total mass. Suppose finally that 
\begin{itemize}
\item $\mathrm{Ric} - \operatorname{Hess}(V) \geq -C \mathrm{Id}$; 
\item there exists a positive $G $-invariant proper function $U \in \calC^2(M)$ such 
that $U+V$ is bounded;  
\item $\|\nabla U\| \rightarrow \infty$ as $U \rightarrow \infty $;
\item $\limsup_{U \rightarrow \infty}  \left(\Delta U/\|\nabla U\|^2 \right)< 1$. 
\end{itemize}
Then there is a strong Hodge decomposition on 
${\rm L}^2_\lambda \Omega^1(M)$, as before. 
\end{theorem}

\begin{cor}\label{GoWa_conditions}
With all notation as above, if $\mathbf{J}$ preserves $\mathcal H^1_\lambda$, and the $G$-action has fixed points, then
it is Hamiltonian.
\end{cor}

\subsection{Further examples}
There are many interesting types of spaces to which the Ahmed-Stroock and Gong-Wang results can be applied, but which 
are not covered by the more classical Theorem \ref{GpTr}. We describe a few of these here, including spaces with asymptotical 
cylindrical or asymptotically conic ends or with complete fibered boundary geometry. Amongst these are the asymptotically locally 
Euclidean (ALE) spaces, as well as the slightly more complicated ALF, ALG, and ALH spaces which arise in the classification of 
gravitational instantons. (We refer to \cite{HHM} for a description of the geometry of ALE/F/G/H spaces.) We can also handle Joyce's 
quasi-ALE (QALE) spaces \cite{Joyce} and their more flexible Riemannian analogues, the quasi-asymptotically conic (QAC) spaces 
of \cite{DegMaz}. The interest in including all of these spaces is that they seem to be intimately intertwined with symplectic geometry; 
indeed,  many of them arise via hyperK\"ahler reduction. 

The obvious idea is to let the function $U$ in Theorem \ref{AScond} depend only on the radial function $r$ on $M$.  Actually, it 
is clear that Theorem \ref{AScond} holds on all of $M$ if and only if it holds on each end (with, say, relative boundary conditions 
on the compact boundaries), so we can immediately localize to each end. We can also replace $g$ on each end by a perhaps 
simpler metric which is quasi-isometric to it.   The general feature of all these spaces is that the distance function $r$ from
a suitably chosen inner boundary has ``symbolic decay properties'', i.e., successively higher derivatives of $r$ decay
increasingly more quickly.   Writing $U = r^a$, then we require that 
\begin{align*}
{\rm i)} & \ \Delta U = a(a-1) r^{a-2} |\nabla r|^2 + a r^{a-1} \Delta r \leq C (1 + r^a) \\
{\rm ii)} & \ \epsilon U^{1+\epsilon} = \epsilon r^{a(1+\epsilon)} \leq 1 + a^2 r^{2a-2} |\nabla r|^2 \\
{\rm iii)} & \ \nabla^2 U = a r^{a-1} \nabla^2 r + a (a-1) r^{a-2} 
\mathbf{d}r^2  \geq -B. 
\end{align*}
Recalling that $|\nabla r| = 1$ holds in general, then ii) implies that $a > 2$, while i) shows
that $\Delta r$ must grow slower than $r$, and finally iii) shows that the level sets $\{r = \mathrm{const.} \}$
have some sort of convexity.  

Rather than trying to determine the most general spaces for which these restrictions hold, we explain why
they are true for the various examples listed above. For the reasons we have explained (namely, that it suffices
to consider a quasi-isometric model), we focus on the simplest models for each of these spaces.  In each
of the following, we consider one end $E$ of $M$. In general we can apply our results to manifolds $M$
which decompose into some compact piece $K$ and a finite number of ends $$E_1, \ldots, E_N,$$ each
of which is of one of the following types. 
 
\medskip

\noindent{\bf Cylindrical ends:}  Here $E = [0,\infty) \times Y$ where $(Y,h)$ is a compact
smooth Riemannian manifold, and $r$ is the linear variable on the first factor.  The metric is the product
$\mathbf{d}r^2 + h$. We obtain conditions i), ii), iii) directly since $\nabla^2 r = 0$. 

\medskip

\noindent{\bf Conic ends:} Now suppose that $E = [1,\infty) \times Y$ where $(Y,h)$ is again
a compact smooth manifold and $r \geq 1$, and the metric is given by 
$g = \mathbf{d}r^2 + r^2 h$. Then 
$$\Delta r = (n-1)/r$$ and $$\nabla^2 r \geq 0,$$ so, once again, all three conditions hold.

\medskip

\noindent{\bf Fibered boundary ends:}  This is slightly more complicated.  Suppose that $Z$ is
a compact smooth manifold which is the total space of a fibration $\pi: Z \to Y$ with fiber $F$. 
Let $h$ be a metric on $Y$ and suppose that $k$ is a symmetric $2$-tensor on $Z$ which
restricts to each fiber $F$ to be positive definite and so that $\pi^* h + k$ is positive definite on $Z$. 
Then  $$E = [1,\infty) \times Z,$$ and 
\[
g = \mathbf{d}r^2 + r^2 \pi^* h + k.
\]
In other words, this metric looks conical in the base ($Y$) directions and cylindrical in the fiber ($F$) directions.
For the specific cases of such metrics that arise in the gravitational instantons above, $Y$ is the quotient of
some $S^k$ by a finite group $\Gamma$ (typically in $\mathrm{SU}(k+1)$) and $F$ is a torus $T^\ell$.
The four-dimensional ALF/ALG/ALH spaces correspond to the cases $$(k,\ell) = (2,1),\ (1,2),\ (0,3).$$ 
The pair $(k,\ell) = (3,0)$ is precisely that of ALE spaces.

For each of these, it is a simple computation to check that $r$ has all the required properties. 

\medskip

\noindent{\bf QALE and QAC ends:}  The geometry of quasi-asymptotically conic spaces
are considerably more difficult to describe in general, and we defer to \cite{Joyce} and \cite{DegMaz} 
for detailed descriptions of the geometry.  These spaces are slightly more complicated in the sense
that while they are essentially conical as $r \to \infty$, the cross-sections $\{r = \mathrm{const.} \}$ 
are families of smooth spaces which converge to a compact stratified space. This is consistent with the
fact that QALE spaces arise as (complex analytic) resolutions of quotients $\mathbb C^n / \Gamma$. 
The basic types of estimates for $r$ and its derivatives are almost the same as above, and so conditions 
i), ii), and iii) still hold. We refer to the monograph and paper cited above for full details. 

\medskip

\noindent{\bf Bundles over QAC ends:}  The final example consists of ends $E$ which are bundles over 
QAC spaces, and with metrics which do not increase the size of the fibers as $r \to \infty$. This is in perfect
analogy to how fibered boundary metrics generalize and fiber over conic metrics.  The behavior of the
function $r$ on these spaces is similarly benign and these same three conditions hold.

\medskip

These examples have been given with very little detail (in the last two cases, barely any). The reason 
for including the, here is because they arise frequently.  In particular, the last category, i.e., bundles over
QAC (or more specifically, QALE) spaces contain the conjectural picture for the important family of 
moduli spaces of monopoles on $\mathbb R^3$.   On none of these spaces is the range of the
Laplacian on unweighted $1$-forms usually closed, but the Ahmed-Stroock conditions provide
an easily applicable way to obtain Hodge decompositions on these spaces.

\section{History of the problem: Frankel's Theorem and further results} \label{HH}

The first result concerning the relationship between the existence of fixed points and the Hamiltonian character of the action is
Frankel's celebrated theorem \cite{Frankel1959} stating that if the manifold
is compact, connected, and K\"ahler, $G=S^1$, and
the symplectic action has fixed points, then it must be Hamiltonian (note that $\mathbf{J}
\mathcal{H} \subset \mathcal{H}$ holds, see \cite[Cor 4.11, Ch. 5]{Wells2008}). Frankel's work has been very influential: 
for example, Ono \cite{Ono1984} proved the analogue theorem for compact Lefschetz manifolds and McDuff \cite[Proposition 2]{McDuff1988} 
has shown that any symplectic circle action on a compact connected symplectic 4-manifold having fixed points is Hamiltonian. 

However, this result fails in higher dimensions: McDuff \cite[Proposition 1]{McDuff1988} gave an example of a compact 
connected symplectic $6$-manifold with a symplectic circle action which has nontrivial fixed point set (equal to a union
of tori), which is nevertheless not Hamiltonian. If the $S^1$-action is semi-free (i.e., free off the fixed point set), then 
Tolman and Weitsman \cite[Theorem 1]{ToWe2000} have shown that any symplectic $S^1$-action on a compact connected 
symplectic manifold having fixed points is Hamiltonian. Feldman \cite[Theorem 1]{Feldman2001} characterized the obstruction 
for a symplectic circle action on a compact manifold to be Hamiltonian and deduced the McDuff and Tolman-Weitsman theorems 
by applying his criterion. He showed that the Todd genus of a manifold admitting a symplectic circle action with isolated
fixed points is equal either to 0, in which case the action is non-Hamiltonian, or to 1, in which
case the action is Hamiltonian. In addition, any symplectic circle action on a manifold with positive Todd genus is Hamiltonian. For additional results regarding aspherical
symplectic manifolds (i.e.  $\varint_{S^2} f ^\ast \omega = 0$ 
for any smooth map $f: S^2 \rightarrow M$) see \cite[Section 8]{KeRuTr2008} and \cite{LuOp1995}.
As of today, there are no known examples of symplectic $S^1$-actions on compact connected symplectic manifolds that are not Hamiltonian but have at least one isolated fixed point.

Less is known for higher dimensional Lie groups. Giacobbe \cite[Theorem 3.13]{Giacobbe2005} proved that a symplectic
action of a $n$-torus on a $2n$-dimensional compact connected symplectic manifold with fixed points is necessarily Hamiltonian;
see also \cite[Corollary 3.9]{DuPe2007}. If $n=2 $ this result can be checked explicitly from the classification 
of symplectic 4-manifolds with symplectic 2-torus actions given in \cite[Theorem 8.2.1]{Pelayo2010} (since cases 2--5 in the 
statement of the theorem are shown not to be Hamiltonian; the only non-K\"ahler cases are given in items 3 and 4 as proved in 
\cite[Theorem 1.1]{DuPe2010}). 

If $G$ is a Lie group with Lie algebra $\mathfrak{g}$ acting symplectically on the symplectic manifold $(M, \omega) $, 
the action is said to be \textit{cohomologically free} if the Lie algebra homomorphism
$$\xi \in \mathfrak{g} \mapsto [\mathbf{i}_{\xi_M} \omega] \in  \operatorname{H}^1(M, \mathbb{R})$$ is 
injective; $\operatorname{H}^1(M, \mathbb{R})$ is regarded as an abelian Lie algebra.  Ginzburg 
\cite[Proposition 4.2]{Ginzburg1992} showed that if a torus $\mathbb{T}^k = (S^1)^k$, $k \in \mathbb{N}$, acts symplectically,
then there exist subtori $\mathbb{T}^{k-r}$, $\mathbb{T}^r$ such that $\mathbb{T}^k =\mathbb{T}^r\times\mathbb{T}^{k-r}$,  
the $\mathbb{T}^r$-action is cohomologically free, and the $\mathbb{T}^{k-r}$-action is Hamiltonian. This homomorphism is the
obstruction to the existence of a momentum map: it vanishes if and only if the action admits a momentum map. For compact
Lie groups the previous result holds only up to coverings. If $G$ is a compact Lie group, then it is well-known that
there is a finite covering $$\mathbb{T}^k \times K \rightarrow G,$$ where $K$ is a semisimple compact Lie group. So
there is a symplectic action of $\mathbb{T}^k \times K$ on $(M, \omega)$. The $K$-action is Hamiltonian, since
$K$ is semisimple. The previous result applied to $\mathbb{T}^k$ implies that there is a finite covering 
$$\mathbb{T}^r\times (\mathbb{T}^{k-r} \times K) \rightarrow G$$ such that the $(\mathbb{T}^{k-r} \times K)$-action is
Hamiltonian and the $\mathbb{T}^r$-action is cohomologically free; this is \cite[Theorem 4.1]{Ginzburg1992}. The Lie
algebra of $\mathbb{T}^{k-r} \times K$ is $\ker\left(\xi\mapsto [\mathbf{i}_{\xi_M} \omega]\right)$. (It appears that the 
argument in \cite{Ginzburg1992} implicitly requires $M$ to satisfy the Lefschetz condition or more generally the flux conjecture to hold for $M$. 
Thus ultimately it depends on \cite{Ono2006} where the flux conjecture is established in full generality. We thank
V. Ginzburg for pointing this out.)
\\
\\
\\
\\
\emph{Acknowledgements.} We thank I. Agol, D. Auroux, J. M. Lee, D. Halpern\--Leistner, X. Tang, and A. Weinstein for many helpful
discussions. We wish to particularly thank A. Weinstein for comments on several preliminary 
versions of the paper. Work on this paper started when the first author was affiliated with the University of California, 
Berkeley (2008\--2010), and the last two authors were members of MSRI.

\newpage
\smallskip\noindent
\\
\\
Rafe Mazzeo\\
Department of Mathematics\\
Stanford University\\
Stanford, CA 94305, USA\\
{\em E\--mail}: \texttt{mazzeo@math.stanford.edu}\\
\\
\\
{\'A}lvaro Pelayo \\
Washington University in St Louis\\
Department of Mathematics \\
One Brookings Drive, Campus Box 1146\\
St Louis, MO 63130-4899, USA.\\
{\em E\--mail}: \texttt{apelayo@math.wustl.edu}\\
\\
\\
\smallskip\noindent
Tudor S. Ratiu\\
Section de Math\'ematiques, Station 8\\
 and Bernoulli Center, Station 14\\
Ecole Polytechnique F\'ed\'erale
de Lausanne\\
 CH-1015 Lausanne, Switzerland\\
{\em E\--mail}: \texttt{tudor.ratiu@epfl.ch}

\begin{thebibliography}{300}

\bibitem[AS00]{AhSt2000}
Ahmed, Z. M., and Stroock, D. W., A Hodge theory for some non-compact manifolds, \textit{J. Diff. Geom.} \textbf{54}
(2000), 177--225.

\bibitem[ALMP12]{ALMP1}
Albin, P., Leichtnam, E., Mazzeo, R., Piazza, P., The signature package on Witt spaces, \textit{Ann. Sci. Ec. Nor. Sup}
(4) {\bf 45} (2012) no. 2, 241--310.
\emph{arxiv:1112.0989}. 

\bibitem[ALMP13]{ALMP2}
Albin, P., Leichtnam, E., Mazzeo, R., Piazza, P., Hodge Theory on Cheeger spaces. \emph{arXiv:1307.5473}.

\bibitem[Ca01]{CarSurvey}
Carron, G., Formes harmoniques ${\rm L}^2$ sur les vari\'et\'es non-compactes. (French) [${\rm L}^2$-harmonic forms on non-compact manifolds] \emph{Rend. Mat. Appl.} {\bf 21} (2001), no. 1-4, 87\--119.

\bibitem[Ca02]{Car2002}
Carron, G., ${\rm L}^2$ harmonic forms on non-compact Riemannian manifolds, in \textit{Surveys in Analysis and Operator Theory} (Canberra, 2001), 49--59, Proc. Centre Math. Appl. Austral. Nat. Univ., \textbf{40}, Austral. Nat. Univ., Canberra, 2002. 

\bibitem[Ch79]{Cheeger}
Cheeger, J.,  On the spectral geometry of spaces with cone-like singularities.
\textit{Proc. Nat. Acad. Sci. U.S.A.} {\bf 76} (1979), no. 5, 2103\--2106. 


\bibitem[DeR84]{dR}
De Rham, G., \textit{Differentiable Manifolds. Forms, Currents, Harmonic Forms}. Translated from the French by F. R. Smith. With an introduction by S. S. Chern. Grundlehren der Mathematischen Wissenschaften [Fundamental Principles of Mathematical Sciences], \textbf{266}, Springer-Verlag, Berlin, 1984.


\bibitem[DM14]{DegMaz}
Degeratu, A. and Mazzeo, R., Fredholm theory on quasi-asymptotically conic spaces, in preparation. 

\bibitem[DP07]{DuPe2007}
Duistermaat, J. J., and Pelayo, A., Symplectic torus actions
with coisotropic orbits, \textit{Ann. Inst. Fourier, Grenoble} \textbf{57} (7) (2007), 2239--2327.

\bibitem[DP11]{DuPe2010}
Duistermaat, J. J. and Pelayo, A., Complex structures on four-manifolds 
with symplectic two-torus actions. 
\emph{Internat. J. Math.} {\bf 22} (2011) 449\--463.

\bibitem[Fe01]{Feldman2001}
Feldman, K. E., Hirzebruch genera of manifolds supporting a Hamiltonian circle action (Russian), \textit{Uspekhi Mat. Nauk} 
 \textbf{56}(5) (2001), 187--188; translation in \textit{Russian Math. Surveys} \textbf{56}(5) (2001), 978--979. 

\bibitem[Fr59]{Frankel1959}
Frankel, T., Fixed points and torsion on K\"ahler manifolds, \textit{Ann. Math.} \textbf{70}(1) (1959), 1--8.


\bibitem[Gi05]{Giacobbe2005}
Giacobbe, A., Convexity of multi-valued momentum maps, \textit{Geometriae Dedicata} \textbf{111} (2005), 1--22.

\bibitem[Gi92]{Ginzburg1992}
Ginzburg, V. L., Some remarks on symplectic actions of compact groups, \textit{Math. Zeitschrift} \textbf{210} (1992), 625--640.

\bibitem[GW04]{GoWa2004} Gong, F.-Z. and Wang, F.-Y., On Gromov's theorem and ${\rm L}^2$-Hodge decomposition, \textit{Int. J. Math. Math. Sci.} \textbf{1--4} (2004), 25\--44.

\bibitem[HHM]{HHM}
Hausel, T., Hunsicker, E., and Mazzeo, R., Hodge cohomology of 
gravitational instantons. 
\textit{Duke Math. J.} {\bf 122} (2004)(3), 485\--548. 

\bibitem[Jo00]{Joyce}
Joyce, D., \textit{Compact Manifolds with Special Holonomy}. Oxford Mathematical Monographs. Oxford University Press, Oxford, 2000. 

\bibitem[KRT08]{KeRuTr2008}
Kedra, J., Rudnyak, Y., and Tralle, A., Symplectically
aspherical manifolds, \textit{Journ. Fixed Point Theory Appl.} 
\textbf{3} (2008), 1--21.

\bibitem[LP95]{LuOp1995}
Lupton, G. and Oprea, J., Cohomologically symplectic spaces: toral 
actions and the Gottlieb group,
\textit{Trans. Amer. Math. Soc.} \textbf{347}(1) (1995), 261--288.



\bibitem[MR03]{MaRa2003}
Marsden, J.E., and Ratiu, T. S., \textit{Introduction to Mechanics and Symmetry}, Texts in Applied Mathematics 
\textbf{17}, second edition, second printing,  Springer Verlag, New York, 2003.

\bibitem[Ma88]{Maz-Hodge}
Mazzeo, R. The Hodge cohomology of a conformally compact metric, 
\emph{J. Differential Geom.} {\bf 28}(2) (1988), 309\--339.

\bibitem[Ma91]{Maz-edge}
Mazzeo, R., Elliptic theory of differential edge operators. I. 
\emph{Comm. Partial Differential Equations} {\bf 16} (1991), 1615\--1664. 


\bibitem[Mc88]{McDuff1988}
McDuff, D., The moment map for circle actions on symplectic
manifolds, \textit{J. Geom. Phys.} \textbf{5}(2) (1988),
149--160.

\bibitem[MS88]{McDSal1998} McDuff, D., and Salamon, D.,  \textit{Introduction to Symplectic Topology}. Oxford Mathematical Monographs, Oxford Science Publications, The Clarendon Press, Oxford University Press, New York, 1995. 

\bibitem[McK70]{McK} McKean, H. P.
An upper bound to the spectrum of $\Delta$ on a manifold of negative curvature.
\emph{J. Differential Geometry} {\bf 4} (1970), 359\--366. 


\bibitem[On84]{Ono1984}
Ono, K., Some remarks on group actions in symplectic geometry,
\textit{Journ. of the Faculty of Science if the University of Tokyo}, Section 1A, 
Mathematics \textbf{35} (1984), 431--437.

\bibitem[On06]{Ono2006} Ono, K.,
Floer-Novikov cohomology and the flux conjecture,
\textit{Geom. Funct. Anal.} \textbf{16}(5) (2006),
981--1020.

\bibitem[Pe10]{Pelayo2010}
Pelayo, A., Symplectic actions of 2-tori on 4-manifolds, 
\textit{Mem. Amer. Math. Soc.}, \textbf{204} (2010) no. 959.

\bibitem[dR73]{deRham1973}
de Rham, G., \textit{Vari\'et\'es diff\'erentiables. Formes, 
courants, formes harmoniques}. Hermann, Paris, 1973.

\bibitem[TW00]{ToWe2000}
Tolman, S. and Weitsman, J., On semifree symplectic circle actions with isolated fixed points, \textit{Topology}, 
\textbf{39} (2000), 299--309.

\bibitem[Tr09]{Tr2009}
Troyanov, M., On the Hodge decomposition in $\mathbb{R}^n$, 
\textit{Moscow Math. J.}, \textbf{9}(4) (2009), 899--926.

\bibitem[We08]{Wells2008}
Wells, R. O., Jr. \textit{Differential Analysis on Complex Manifolds}, 
third edition, Graduate Texts in Mathematics, 
\textbf{65}, Springer, New York, 2008.


\end{thebibliography}
\end{document}